\documentclass[12pt]{amsart}
\usepackage{latexsym, amsthm, amscd, euscript, float, subfig}
\usepackage{color}

\usepackage{amssymb}
\usepackage{amsmath}
\usepackage{graphicx}

{}

\newtheorem{theorem}{Theorem}[section]
\newtheorem{lemma}[theorem]{Lemma}
\newtheorem{corollary}[theorem]{Corollary}
\newtheorem{proposition}[theorem]{Proposition}

\newtheorem{remark}[theorem]{Remark}

\newtheorem*{acknowledgement}{Acknowledgment}

\newtheorem{conjecture}[theorem]{Conjecture}

\numberwithin{equation}{section}

\newcounter{minutes}
\divide\time by 60
\newcounter{hours}
\multiply\time by 60
\addtocounter{minutes}{-\time}

\newcommand{\IR}{\mathbb{R}}
\newcommand{\IB}{\mathbb{B}}
\newcommand{\ds}{\displaystyle}

\newcommand{\Bn}{ {\mathbb{B}^n} }
\newcommand{\Rn}{ {\mathbb{R}^n} }
\newcommand{\Hn}{ {\mathbb{H}^n} }
\newcommand{\Ht}{ {\mathbb{H}^2} }

\newcommand{\beq}{\begin{equation}}
\newcommand{\eeq}{\end{equation}}
\newcommand{\arsh}{\,\textnormal{arsh}}

\renewcommand{\sinh}{\,\textnormal{sinh}}

\begin{document}
\vspace*{-2cm}

\title[Geometric properties of the Cassinian metric]
{Geometric properties of the Cassinian metric}

\def\thefootnote{}
\footnotetext{ \texttt{\tiny File:~\jobname .tex,
          printed: \number\day-\number\month-\number\year,
          \thehours.\ifnum\theminutes<10{0}\fi\theminutes}
} \makeatletter\def\thefootnote{\@arabic\c@footnote}\makeatother

\author[R. Kl\'en]{Riku Kl\'en}
\address{Riku Kl\'en, Department of Mathematics and Statistics,
University of Turku, Turku 20014, Finland 
}
\email{riku.klen@utu.fi}

\author[M. R. Mohapatra]{Manas Ranjan Mohapatra}
\address{Manas Ranjan Mohapatra, Discipline of Mathematics,
Indian Institute of Technology Indore,
Simrol, Khandwa Road, Indore 452 020
}
\email{mrm.iiti@gmail.com}

\author[S. K. Sahoo]{Swadesh Kumar Sahoo}
\address{Swadesh Kumar Sahoo, Discipline of Mathematics,
Indian Institute of Technology Indore,
Simrol, Khandwa Road, Indore 452 020}
\email{swadesh@iiti.ac.in}

\begin{abstract}
In this paper we 
prove a sharp distortion property of the Cassinian metric under M\"obius transformations 
of a punctured ball onto another punctured ball. 
The paper also deals with discussion on local convexity properties of the 
Cassinian metric balls in some specific subdomains of $\mathbb{R}^n$. Inclusion properties of 
the Cassinian metric balls with other hyperbolic-type metric balls are also
investigated. 
In particular, several conjectures are also stated in response to sharpness 
of the inclusion relations.
\\

\smallskip
\noindent
{\bf 2010 Mathematics Subject Classification}. 30C35, 30C20, 30F45, 51M10.

\smallskip
\noindent
{\bf Keywords and phrases.}
M\"obius transformations, Cassinian metric, 
distortion property, inclusion property, metric balls, convexity.
\end{abstract}

\maketitle

\section{Introduction}\label{Intro}
Conformal invariants have important roles in geometric function theory. One of
the basic conformal invariants is the modulus of a curve family, which is
used to define quasiconformal maps \cite{Vai71}. 
In Euclidean spaces of dimension more than two, 
conformal maps are nothing but the restrictions of M\"obius maps; for instance, see \cite{Vai71,Vuo07}. 
Therefore, it is natural to study M\"obius invariants in the higher dimensional setting. 
There are some metrics which are M\"obius invariant 
and some are not. For example, the Apollonian \cite{Bea95,Has03} and the Seittenranta \cite{Sei99} metrics are M\"obius 
invariant whereas the quasihyperbolic \cite{GO79,GP76} and the distance ratio \cite{Vuo88} metrics are not. 
The study of the M\"obius quasi-invariance property is hence natural for these metrics which 
are not M\"obius invariant.
In other words, it would be interesting if we obtain the Lipschitz constants for those metrics which do not belong to
the M\"obius invariant family.
Gehring and Osgood in \cite{GO79} proved that the distance ratio metric and the quasihyperbolic metric 
are not changed by more than a factor $2$ under M\"obius maps. In \cite{SV,SVW}, sharp distortion properties of the distance ratio metric under M\"obius transformations of the unit (punctured) balls are obtained. A sharp distortion property of the Cassinian metric
under M\"obius transformations of the unit ball is also recently established in \cite{IMSZ}.  

Most of the metrics mentioned in this paper belong to the family of relative metrics,
some are M\"obius invariant and some are not. By a {\em relative metric} we mean a 
metric that is evaluated in a proper subdomain of $\IR^n$ relative to its boundary.
Note that the topological operations (boundary, closure, etc.) are taken in the compact space $\overline{\IR^n}$.
One of our main objectives in this paper is to consider a relative metric, a special case of the family of metrics
defined in \cite[Lemma~6.1]{Has02}, the so-called {\em Cassinian metric}, independently 
first studied by Ibragimov \cite{Ibr09} and look at its M\"obius quasi-invariance properties.
Note that the generalized relative metric defined in \cite[Lemma~6.1]{Has02} is named as 
the {\em $M$-relative metric} and defined on a domain $D\subsetneq \IR^n$ by the quantity
\begin{equation}\label{M-relative}
\rho_{M,D}(x,y):=\sup_{a\in\partial D}\frac{|x-y|}{M(|x-a|,|y-a|)}, 
\end{equation}
where $M$ is continuous in $(0,\infty)\times (0,\infty)$ and $\partial D$ denotes the boundary of $D$. 
If $M(\alpha,\beta)=\alpha\beta$,
then the corresponding relative metric $\rho_{M,D}$ defines the Cassinian metric introduced 
in \cite{Ibr09} and subsequently studied in \cite{HKVZ15,IMSZ}. 
The choice $M(\alpha,\beta)=\alpha+\beta$ similarly leads to the triangular ratio metric
recently investigated in \cite{CHKV15,HKLV15}. 
We refer to the next section for the explicit definition of 
the Cassinian metric.

In one hand, we study distortion properties of the Cassinian metric under M\"obius and bi-Lipschitz maps
in Section~\ref{sec3}.
On the other hand, we also focus on a general question suggested by Vuorinen in \cite{Vuo07}
about the convexity of balls of small radii in metric spaces. This problem has been investigated
by Kl{\'e}n in different situations in a series of papers \cite{Klej08,Kleq08,Kle10,Kle13}. 
In this context, we study convexity properties of the Cassinian metric balls
in Section~\ref{sec4}. 
Section~\ref{sec5} is devoted to the inclusion properties of the Cassinian metric balls with other related metric balls.

\section{Common notation and terminology}\label{sec2}
Throughout the paper we use the notation $\IR^n,n\ge 2,$ for the Euclidean 
$n$-dimensional space; $\overline{\IR^n}:=\IR^n\cup \{\infty\}$ for the 
one point compactification of $\IR^n$. 
The Euclidean distance between $x,y\in\IR^n$ is denoted by $|x-y|$. 
Given $x\in\IR^n$ and $r>0$, the open ball centered at $x$ and radius $r$ is 
denoted by $B^n(x,r):=\{y\in\IR^n\colon\, |x-y|<r\}$. Denote by $\IB^n:=\IB^n(0,1)$, the unit ball in $\IR^n$.
Consequently, we set $\mathbb{H}^n:=\{x=(x_1,x_2,\ldots,x_n)\in\mathbb{R}^n:\,x_n>0\}$,
the upper half-space. 

\subsection*{The Cassinian metric}
Let $D\subsetneq \IR^n$ be an arbitrary domain. 
The {\it Cassinian metric}, $c_D$, on $D$ is defined by
$$c_D (x,y)=\sup_{p\in \partial D} \frac{|x-y|}{|x-p||p-y|} .
$$
Note that the quantity $c_D$ defines a metric on $D$; see \cite[Lemma~3.1]{Ibr09}.
Geometrically, the Cassinian metric can be defined in terms of maximal Cassinian ovals 
(see \cite[Sections 3.1-3.2]{Ibr09} and references therein) in 
the domain $D$ in a similar fashion as the Apollonian metric is defined in terms of 
maximal Apollonian balls \cite{Bea98}. 

We end this section with the definitions of the hyperbolic metric, 
the quasihyperbolic metric and the distance ratio metric used in the subsequent sections.

\subsection*{The hyperbolic metric}
The {\em hyperbolic metric} of the unit ball 
$\Bn$ is defined by
$$ \rho_{\Bn}(x,y)=\inf_{\gamma\in \Gamma(x,y)} \int_{\gamma}\frac{2|dz|}{1-|z|^2},
$$
where $\Gamma(x,y)$ denotes the family of rectifiable curves joining $x$ and $y$ in $\Bn$.

\subsection*{The quasihyperbolic metric}
Let $D\subsetneq \IR^n$ be an arbitrary domain.
The {\em quasihyperbolic metric} \cite{GP76} is defined by
$$k_D(x,y)=\inf_{\gamma\in \Gamma(x,y)} \int_{\gamma} \frac{|dz|}{\delta_D(z)},
$$
where $\Gamma(x,y)$ denotes the family of rectifiable curves joining $x$ and $y$ in $D$
and $\delta_D(z)={\rm dist}\,(z,\partial D)$, the shortest Euclidean distance from 
$z$ to $\partial D$.
The quasihyperbolic metric was introduced by Gehring and Palka in 1976 and subsequently studied by
Gehring and Osgood; see \cite{GO79,GP76}, as a generalization of the
hyperbolic metric of the upper half plane to arbitrary proper subdomains of $\Rn$.

\subsection*{The distance ratio metric}
Let $D\subsetneq \IR^n$. For any two points $x,y\in D$, the {\em distance ratio metric}, 
$j_D(x,y)$, is defined as
$$j_D(x,y)=\log \left(1+\frac{|x-y|}{\delta_D(x)\wedge\delta_D(y)}\right),
$$  
where $\delta_D(x)\wedge\delta_D(y)=\min\{\delta_D(x),\delta_D(y)\}$. This form of the metric $j_D$, 
which was first considered by
Vuorinen in \cite{Vuo85}, is a slight modification of the original distance ratio metric 
introduced by Gehring and Osgood in \cite{GO79}. 
This metric has been widely studied in the literature; see, for instance, \cite{Vuo88}.


\section{Distortion Property of the Cassinian Metric under M\"obius Transformations}\label{sec3}

One of our objectives in this section is to study the distortion property of the Cassinian metric 
under M\"obius maps from a punctured ball onto another punctured ball.
Distortion properties of the Cassinian metric of the unit ball under M\"obius maps has been
recently studied in \cite{IMSZ}.
\begin{theorem}\label{lip-bn}
Let $a\in \mathbb{B}^n$ and $f:\mathbb{B}^n\setminus\{0\} \to \mathbb{B}^n\setminus\{a\}$ 
be a M\"obius map with $f(0)=a$. Then for $x,y\in \mathbb{B}^n\setminus\{0\}$ we have
$$ \frac{1-|a|}{1+|a|}c_{\mathbb{B}^n\setminus \{0\}}(x,y)\le c_{\mathbb{B}^n\setminus\{a\}}(f(x),f(y))\le \frac{1+|a|}{1-|a|}
c_{\mathbb{B}^n\setminus\{0\}}(x,y).
$$
The equalities in both sides can be attained.
\end{theorem}
\begin{proof}
If $a=0$, the proof is trivial (see \cite{IMSZ}). 
Now, assume that $a\neq 0$. Let $\sigma$ be the inversion in the sphere $\mathbb S^{n-1}(a^\star,r)$, where 
$$
a^\star=\frac {a}{|a|^2}\qquad\text{and}\qquad r=\sqrt{|a^\star|^2-1}=\frac {\sqrt{1-|a|^2}}{|a|}.
$$
Note that the sphere $\mathbb S^{n-1}(a^\star,r)$ is orthogonal to $\mathbb{S}^{n-1}$ and that $\sigma(a)=0$. 
In particular, $\sigma$ is a M\"obius map with $\sigma(\IB^n\setminus\{a\})=\IB^n\setminus\{0\}$. 
Recall from \cite{Bea95} that 
\begin{equation}\label{inversionmap}
\sigma(x)=a^\star+\Big(\frac {r}{|x-a^\star|}\Big)^2\big(x-a^\star\big).
\end{equation}
Then $\sigma\circ f$ is an orthogonal matrix (see, for example, \cite[Theorem 3.5.1(i)]{Bea95}). 
In particular,  
\begin{equation}\label{mob1}
\Big|\sigma\big(f(x)\big)-\sigma\big(f(y)\big)\Big|=|x-y|.
\end{equation}
We will need the following property of $\sigma$ (see, for example, \cite[p. 26]{Bea95}):
\begin{equation}\label{mob2}
|\sigma(x)-\sigma(y)|=\frac {r^2|x-y|}{|x-a^\star||y-a^\star|}.
\end{equation}
Now, 
$$c_{\mathbb{B}^n\setminus\{0\}}(x,y)=\frac{|x-y|}{\min\{|x||y|,
\inf_{z\in \partial \IB^n}|x-z||z-y|\}}
$$
and
$$c_{\mathbb{B}^n\setminus\{a\}}(f(x),f(y))=\frac{|f(x)-f(y)|}{\min\{|f(x)-a||a-f(y)|,
\inf_{w\in \partial \IB^n}|f(x)-w||w-f(y)|\}}.
$$
Denote by $P=\min\{|f(x)-a||a-f(y)|,\inf_{w\in \partial \IB^n}|f(x)-w||w-f(y)|\}$. 
Now we have two choices for $P$.

\noindent
{\bf Case I.} $P=|f(x)-a||a-f(y)|$.

\noindent It follows from (\ref{mob1}) and (\ref{mob2}) that
$$|f(x)-a|=\frac{|f(x)-a^\star||a-a^\star|}{(|a^\star|^2-1)}|x| ~~\mbox{ and }~~
|a-f(y)|=\frac{|f(y)-a^\star||a-a^\star|}{(|a^\star|^2-1)}|y|.
$$
Now, 
$$c_{\mathbb{B}^n\setminus\{a\}}(f(x),f(y))=\frac{|x-y|}{|x||y|}.
\frac{|a^\star|^2-1}{|a-a^\star|^2}\le
\frac{1}{1-|a|^2}\,c_{\mathbb{B}^n\setminus\{0\}}(x,y).
$$

\noindent
{\bf Case II.} $P=\inf_{w\in \partial \IB^n}|f(x)-w||w-f(y)|$.

\noindent Proof of this case follows from the proof of \cite[Theorem~4.1]{IMSZ}.
The lower estimate follows by considering the inverse mapping $f^{-1}$.

To see the sharpness, consider the map $\sigma$ defined by (\ref{inversionmap}).
For $0<s<t<1$, choose the points $u=-te_1$ and $v=-se_1$ in such a way that
$$c_{\IB^n\setminus\{0\}}(x,y)=\frac{t-s}{(1-t)(1-s)}
~~\mbox{ and }~~
c_{\IB^n\setminus\{a\}}(\sigma(x),\sigma(y)) = \frac{|\sigma(x)-\sigma(y)|}{(1-|\sigma(x)|)(1-|\sigma(y)|)}.
$$
The image points of $x$ and $y$ under $\sigma$ is given by
$$\sigma(x)=\frac{|a|+t}{1+|a|t}e_1,\quad \sigma(y)=\frac{|a|+s}{1+|a|s}e_1\in [a,e_1]\setminus\{a,e_1\}.
$$
Now, the Cassinian distance between $\sigma(x)$ and $\sigma(y)$ is
\begin{eqnarray*}
c_{\IB^n\setminus\{a\}}(\sigma(x),\sigma(y)) &=& \frac{|\sigma(x)-\sigma(y)|}{(1-|\sigma(x)|)(1-|\sigma(y)|)}\\
&= & \frac{\left|\ds\frac{|a|+t}{1+|a|t}-\frac{|a|+s}{1+|a|s}\right|}{\left(1-\left|\ds\frac{|a|+t}{1+|a|t}\right|\right)
\left(1-\left|\ds\frac{|a|+s}{1+|a|s}\right|\right)}\\
&=& \frac{t-s}{(1-t)(1-s)}.\frac{1+|a|}{1-|a|}\\
&=& \frac{1+|a|}{1-|a|}c_{\IB^n\setminus\{0\}}(x,y). 
\end{eqnarray*}
The lower bound can be seen by considering the inverse of $\sigma$ and hence the conclusion follows.
\end{proof}
\begin{remark}
Gehring and Osgood proved that the quasihyperbolic and the distance ratio metrics are not changed by more than a factor $2$ under M\"obius
transformations (see \cite[Theorem~4]{GO79}).
Naturally, one can ask the similar question for the Cassinian metric. Unfortunately, such a finite constant does not exist for the Cassinian metric in arbitrary proper subdomains of $\Rn$. 
Indeed, the reason is clear from Theorem~\ref{lip-bn} when $|a|\to 1$ (see also \cite[Theorem~4.1]{IMSZ}). However, if we replace M\"obius mappings by bi-Lipschitz mappings of $\Rn$, the following fact guarantees that distortion constant exists.

\noindent {\bf Fact.} Let $f:\Rn\to \Rn$ be an $L$-bilipschitz mapping, that is
$$|x-y|/L \le |f(x)-f(y)|\le L|x-y|
$$
for all $x,y\in \Rn$, which maps $D\subsetneq \Rn$ onto $D'\subsetneq \Rn$.
Then 
$$\frac{1}{L^3} c_{D}(x_1,x_2)\le c_{D'}(f(x_1),f(x_2))\le L^3 c_D(x_1,x_2)
$$
for all $x_1,x_2\in D$.

\end{remark}

\section{Convexity properties of Cassinian metric balls}\label{sec4}
This section focuses on the local convexity
properties of the Cassinian metric ball. 
We define the metric ball as follows: let $(D,d)$
be a metric space. Then the set 
$$B_d(x,R)=\{z\in D:\,d(x,z)<R\}
$$
is called the {\em $d$-metric ball} of the domain $D$. 
A metric ball with respect to the Cassinian metric is called 
a {\em Cassinian (metric) ball}.

Before entering into the discussion on the convexity properties,
we describe the Cassinian ball of a domain $D$
in terms of Cassinian balls of $\mathbb{R}^n\setminus \partial D$ fixing a centre in $D$.
The following proposition is a consequence of the proof of \cite[Theorem~1.1]{Klej08}
with respect to the Cassinian metric. 

\begin{proposition}
Let $D\subsetneq \Rn$ and $x\in D$. Then 
$$B_{c_D}(x,R)=\ds \cap_{z\in \partial D} B_{c_{\Rn\setminus\{z\}}}(x,R).
$$
\end{proposition}

\begin{proof}
Suppose that $y\in \ds\cap_{z\in \partial D} B_{c_{\Rn\setminus\{z\}}}(x,R)$. Then 
$c_{\Rn\setminus\{z\}}(x,y)<R$ for all $z\in \partial D$. Choose $z'\in \partial D$
such that
$$c_{\Rn\setminus\{z'\}}(x,y)=\min_{z\in \partial D}c_{\Rn\setminus\{z\}}(x,y).
$$
As $z'\in \partial D$, it is clear that $c_D(x,y)\le c_{\Rn\setminus\{z'\}}(x,y)$ and 
hence $c_D(x,y)<R$. 
Hence $B_{c_D}(x,R)\supseteq \ds\cap_{z\in \partial D} B_{c_{\Rn\setminus\{z\}}}(x,R)$.
Conversely, let $y\in B_{c_D}(x,R)$ and suppose that $z^*\in \partial D$ such that 
$$|x-z^*||z^*-y|=\inf_{z\in \partial D}|x-z||z-y|.
$$
Clearly, $c_{\Rn\setminus\{z'\}}(x,y)(x,y)\le c_D(x,y)<R$. Hence 
$B_{c_D}(x,R)\subseteq \ds\cap_{z\in \partial D} B_{c_{\Rn\setminus\{z\}}}(x,R)$ and the 
proof is complete.
\end{proof}

Let $D\subset \IR^n$ be a domain. We say that $D$ is 
{\em convex} if the line segment $[x,y]$ joining any pair of points $x$ and $y$ 
entirely contained in $D$. A domain $D$ is {\em strictly convex} 
if $[x',y']\cap \partial D=\{x',y'\}$ for any pair of points 
$x',y'\in \partial D$.

We now begin with studying local convexity properties of the Cassinian ball.
For $n=2$, we call these Cassinian balls as {\em the Cassinian disks}. 

\begin{theorem}\label{loc-con}
Let $x\in \mathbb{R}^2\setminus \{0\}$. Then
\begin{itemize}
\item[{\bf (a)}] the Cassinian disk $B_{c_{\mathbb{R}^2\setminus \{0\}}}(x,R)$ is 
convex if and only if $R\in (0,1]$.
\item[{\bf (b)}] the Cassinian disk $B_{c_{\mathbb{R}^2\setminus \{0\}}}(x,R)$ is 
strictly convex if and only if $R\in (0,1)$.
\end{itemize}
\end{theorem}
\begin{proof}
{\bf (a)}
Without loss of generality we can assume that $x=1$. 
Let $z$ be an arbitrary point in $\partial B_{c_{\mathbb{R}^2\setminus\{0\}}}(1,R)$.
Consider the circle (with respect to the Cassinian metric) 
$\partial B_{c_{\mathbb{R}^2\setminus\{0\}}}(1,R)$ for a fixed $R$.
From the definition of the Cassinian disk, it follows that 
the boundary $\partial B_{c_{\mathbb{R}^2\setminus\{0\}}}(1,R)$ is an Euclidean circle with
center $1/(1-R^2)$ and radius $R/(1-R^2)$. Therefore, $\partial 
B_{c_{\mathbb{R}^2\setminus\{0\}}}(1,R)$ is convex for $R\le 1$ and not convex for $R>1$.

{\bf (b)} When $R=1$ the center of the above Euclidean circle 
becomes $\infty$, and thus $\partial B_{c_{\mathbb{R}^2\setminus\{0\}}}(x,R)$
is not strictly convex.
\end{proof}
\begin{corollary}Let $x\in \mathbb{R}^n\setminus \{0\}$. Then
\begin{itemize}
\item[{\bf (a)}] the Cassinian ball $B_{c_{\mathbb{R}^n\setminus \{0\}}}(x,R)$ is 
convex if and only if $R\in (0,1]$.
\item[{\bf (b)}] the Cassinian ball $B_{c_{\mathbb{R}^n\setminus \{0\}}}(x,R)$ is
strictly convex if and only if $R\in (0,1)$.
\end{itemize}

\end{corollary}

\begin{figure}[H]
\begin{center}
\includegraphics[height=7cm,width=6.5cm]{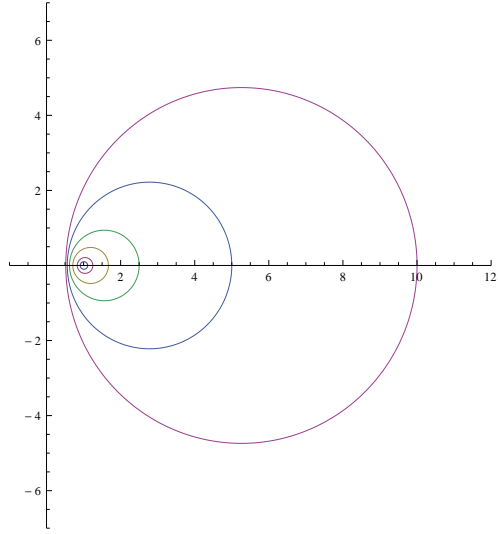}
\caption{The Cassinian disks 
$B_{c_{\mathbb{R}^2\setminus\{0\}}}(1,R)$ with radii $R=0.1,0.2,0.4,0.6,0.8$ and $0.9$.}
\end{center}
\end{figure}

In the  punctured space the Cassinian balls are convex with small radius, but the same is not true in general. The next result shows that this is not the case even in convex domain.

\begin{proposition}\label{local conv Hn}
  Let $r > 0$. There exists $x \in \Hn$ such that $B_{c_\Hn}(x,r)$ is not convex.
\end{proposition}
\begin{proof}
  It is sufficient to consider the case $n=2$. For a given $r$ we choose $x=i/r$ and consider Cassinian disk $B_{c_\Ht}(x,r)$ with radius $r$. To show that $B_{c_\Ht}(x,r)$ is not convex we choose two points $y_1$ and $y_2$ such that $c_\Ht(x,y_1)=r=c_\Ht(x,y_2)$ and $c_\Ht(x,(y_1+y_2)/2)>r$.
  
  We choose $y_1 = \frac{i}{2r}$. Now by the geometry of the Cassinian ovals
  \[
    c_\Ht(x,y_1) = \frac{|x-y_1|}{|x-0||0-y_1|} = \frac{\frac{1}{2r}}{\frac{1}{r}\frac{1}{2r}} = r.
  \]
  
  Let $y_2 = (2+i)/r$. Again by the geometry of the Cassinian ovals
  \[
    c_\Ht(x,y_2) = \frac{|x-y_2|}{|x-1/r||1/r-y_2|} = \frac{\frac{2}{r}}{\sqrt{1/r^2+1/r^2} \sqrt{1/r^2+1/r^2}} = r.
  \]
  
  Now $y_3=(y_1+y_2)/2 = \frac{1}{r}+\frac{3i}{4r}$ and we choose $z=\frac{2}{3r}$. We obtain
  \[
    |x-y_3| = \frac{\sqrt{17}}{4r}, \quad |x-z|=\frac{\sqrt{13}}{3r}, \quad |z-y| = \frac{\sqrt{97}}{12r}
  \]
  and thus
  \[
    c_\Ht(x,y_3) \ge \frac{|x-y_3|}{|x-z||z-y_3|} = 9 \sqrt{ \frac{17}{1261} } r = 1.04498 \dots r > r. \qedhere
  \]
\end{proof}

\begin{figure}[H]
\begin{center}
\includegraphics[width=10cm]{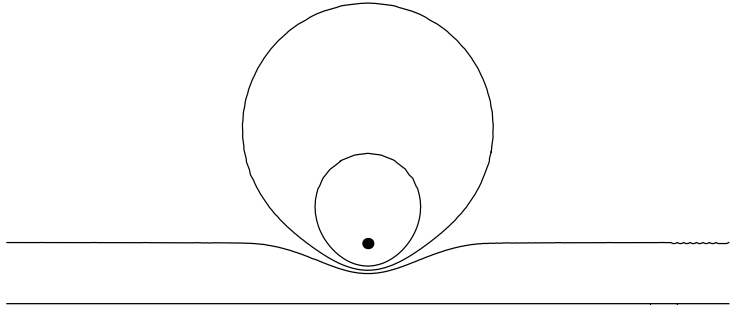}
\caption{The Cassinian disks 
$B_{c_{\Ht}}(i,R)$ with radii $R=0.6,0.8$ and $1.0$.}
\end{center}
\end{figure}

The proof of Proposition \ref{local conv Hn} suggest that the radius of convexity for the Cassinian balls $B_{c_\Ht}(x,r)$ in $\Hn$ depends on $\delta_\Ht (x)$ and $r$. We pose the following conjecture.
\begin{conjecture}\label{convexity conj 1}
  There exists $r_0 > 0$ such that $B_{c_\Ht}(x,r)$ is convex for all $x \in \Ht$ and $r \in (0,a \delta_\Ht (x)]$.
\end{conjecture}
Computer experiments suggest that Conjecture \ref{convexity conj 1} holds true for constant $a \approx 0.85$. We pose also some more general conjectures.
\begin{conjecture}
  In a bounded convex domain $G$, $B_{c_G}(x,r)$ is convex for all $x \in G$ and $r>0$.
\end{conjecture}
\begin{conjecture}
  Let $r>0$. If $G$ is a starlike domain with respect to $x \in G$, then $B_{c_G}(x,r)$ is starlike with respect to $x$.
\end{conjecture}

\section{Inclusion properties}\label{sec5}
The purpose of this section is to study inclusion properties of the Cassinian metric balls
in proper subdomains of $\mathbb{R}^n$ with other related metric balls. 
In other words, we consider the problems of the following type: 
{\em for given $x\in D\subsetneq \IR^n$ and $t>0$, we find optimal radii $r,R>0$ 
depending only on $x$ and $t$ such that
\begin{equation}\label{inclusion}
B_{d}(x,r)\subset B_{c}(x,t)\subset B_{d}(x,R),
\end{equation}
where $d$ is a metric other than the Cassinian metric defined on $D$.}

We begin with proving the relation (\ref{inclusion}) when $d$ is the Euclidean metric.
\begin{theorem}\label{c+euclidean}
Let $D\subsetneq \IR^n$ and $x\in D$. Assume that $1-t\delta_D(x)\neq 0$ for $t>0$.
Then the following inclusion property holds:
$$B^n(x,r)\subset B_c(x,t)\subset B^n(x,R)
$$
where $r=\ds\frac{t(\delta_D(x))^2}{1+t\delta_D(x)}$ and $R=\ds\frac{t(\delta_D(x))^2}{1-t\delta_D(x)}$. The radii $r$ and $R$ are best possible. Moreover, $R/r\to 1$ as $t\to 0$.
\end{theorem}

\begin{proof}
It is clear that for $0\neq x\in D$,
$$\inf_{z\in \partial D}|x-z||z-y|\le \delta_D(x)(\delta_D(x)+|x-y|) .
$$
By the definition of the Cassinian metric $c_D$ we have
$$c_D(x,y)=\frac{|x-y|}{\inf_{z\in \partial D}|x-z||z-y|}\ge \frac{|x-y|}{\delta_D(x)(\delta_D(x)+|x-y|)}
$$
which implies
$$|x-y|\le \frac{c_D(x,y)(\delta_D(x))^2}{1-c_D(x,y)\delta_D(x)}.
$$
Hence the second inclusion holds. Now we prove the first inclusion. Let $y\in B^n(x,t)$.
Then $y\in B^n(x,\delta_D(x))$ and by the monotone property 
$$c_D(x,y)\le c_{B^n(x,\delta_D(x))}(x,y)=\ds\frac{|x-y|}{\delta_D(x)(\delta_D(x)-|x-y|)}.
$$
In particular, if $y\in B^n(x,r)$ with $r=t(\delta_D(x))^2/(1+t\delta_D(x))$, 
then $y\in B_c(x,t)$. Clearly, one can see that 
$$\frac{R}{r}=\frac{1+t\delta_D(x)}{1-t\delta_D(x)}\to 1 \mbox{ as } t\to 0.
$$ 
We finally show that radii $r$ and $R$ are best possible. For this, consider the domain 
$D=\Rn\setminus\{a\}$ and $x\in D$. Let us denote by $l$ the line through points $a$ and $x$. We set $\{ y_1,y_2 \} = \partial B_c(x,t) \cap l$ with $|a-y_1| < |a-y_2|$.
Now
\[
  c_D(x,y_1) = t = \frac{|x-y_1|}{|x-a|(|x-a|-|x-y_1|)}
\]
implying
\[
  |x-y_1| = \frac{t|x-a|^2}{1+t |x-a|}
\]
which shows that $r$ is best possible. Similarly,
\[
  c_D(x,y_2) = t = \frac{|x-y_2|}{|x-a|(|x-a|+|x-y_2|)}
\]
implying
\[
  |x-y_2| = \frac{t|x-a|^2}{1-t |x-a|}
\]
which shows that $R$ is best possible. This completes the proof of our theorem.
\end{proof}



Now we move forward to discuss the inclusion relation (\ref{inclusion}) when
$d=j_D$. 
Recall the following relation proved in \cite{Sei99}:
\begin{lemma}\cite[Theorem~3.8]{Sei99}\label{jb-inclusion}
If $D\subsetneq \Rn$ is open, $x\in D$ and $t>0$, then
$$B^n(x,r)\subset B_j(x,t)\subset B^n(x,R)
$$
where $r=(1-e^{-t})\delta_D(x)$ and $R=(e^t-1)\delta_D(x)$. The formulas for $r$ and $R$
are the best possible expressed in terms of $t$ and $\delta_D(x)$ only.
\end{lemma}
In this connection, we prove
\begin{theorem}\label{c-j inclusion}
Let $D\subsetneq \Rn$, $x\in D$ and $t>0$. Then the following holds:
$$B_j(x,r)\subset B_c(x,t)\subset B_j(x,R),
$$
where $r=\ds\log\left(1+\frac{t\delta_D(x)}{1+t\delta_D(x)}\right)$ 
and $R=\ds\frac{t\delta_D(x)}{1-t\delta_D(x)}$. Moreover, $R/r\to 1$ as $t\to 0$.
\end{theorem}
\begin{proof}
We first prove the second inclusion. By \cite[Theorem~3.4]{IMSZ} we have 
$$j_D(x,y)\le (|x-y|+\delta_D(x)\wedge \delta_D(y))c_D(x,y)\le (|x-y|+\delta_D(x))c_D(x,y)
$$
and from Theorem~\ref{c+euclidean},
$$c_D(x,y)<t \implies |x-y|<t(\delta_D(x))^2/(1-t\delta_D(x)).
$$ 
Now for $y\in B_c(x,r)$,
using the above estimates we have, $j_D(x,y)<t\delta_D(x)/(1-t\delta_D(x))$.  
For the proof of the first inclusion we use Lemma~\ref{jb-inclusion} together with Theorem~\ref{c+euclidean} to conclude that 
$$j_D(x,y)<\log(1+(r\delta_D(x))/(1+r\delta_D(x)))\implies c_D(x,y)<r.
$$
By l'H$\hat{\mbox{o}}$spital rule it follows that $R/r\to 1$ as $t\to 0$.
\end{proof}


The radii obtained in Theorem~\ref{c-j inclusion} can be improved in the special case if we choose the domain $D=\Rn\setminus\{a\},a\in \Rn$. In this connection we prove

\begin{theorem}\label{c+euclidean+punctured}
For $a\in \Rn$, let $D=\Rn\setminus\{a\}$, $x\in D$ and $t>0$. Then the following holds:
$$B_j(x,r)\subset B_c(x,t)\subset B_j(x,R),
$$
where $r=\ds\log (1+t|x-a|)$ 
and $R=\log\left(\ds\frac{1}{1-t|x-a|}\right)$. The radii $r$ and $R$ are best possible. Moreover, $R/r\to 1$ as $t\to 0$.
\end{theorem} 

\begin{proof}
Suppose that $y\in B_j(x,r)$. Then $j_D(x,y)<r$. On simplification, we get 
\begin{equation}\label{eqn5.7.1}
|x-y|<t|x-a|(|x-a|\wedge |y-a|)
\end{equation}
If $|x-a|\wedge |y-a|=|x-a|$, then
$$c_D(x,y)=\frac{|x-y|}{|x-a||y-a|}<\frac{t|x-a|}{|y-a|}\le t
$$
where the first inequality follows from (\ref{eqn5.7.1}) and the last inequality follows
from the fact that $|x-a|\le |y-a|$. Otherwise, 
$$c_D(x,y)=\frac{|x-y|}{|x-a||y-a|}<\frac{t|x-a||y-a|}{|x-a||y-a|}=t
$$
where the inequality follows from (\ref{eqn5.7.1}) and the first inclusion follows.
Now suppose that $c_D(x,y)<t$. This implies, By Theorem~\ref{c+euclidean}, $|x-y|<t|x-a|^2/(1-t|x-a|)$. If $|x-a|\wedge |y-a|=|x-a|$, then
$$j_D(x,y)<\log\left(1+\ds\frac{t|x-a|}{1-t|x-a|}\right)=\log\left(\ds\frac{1}{1-t|x-a|}
\right).
$$
Otherwise
$$j_D(x,y)<\log\left(1+\ds\frac{t|x-a|^2}{|y-a|(1-t|x-a|)}\right)\le \log\left(
\frac{1-t|x-a|}{1-2t|x-a|}\right)
$$
where the second inequality follows from the fact that
$$|y-a|\ge |x-a|-|x-y|\ge \frac{|x-a|(1-2r|x-a|)}{1-r|x-a|}.
$$
Now,
$$R=\min\left \{ \log\left(\ds\frac{1}{1-t|x-a|}\right), \log\left(
\frac{1-t|x-a|}{1-2t|x-a|}\right)\right \}=\log\left(\ds\frac{1}{1-t|x-a|}\right)
$$
and hence the proof of the second inclusion follows. 
By l'H$\hat{\mbox{o}}$spital rule it follows that
$$\lim_{t\to 0} \frac{R}{r}=\lim_{t\to 0} \frac{1+t|x-a|}{1-t|x-a|}=1.
$$

To show that radii $r$ and $R$ are best possible, we consider the same construction as did
in the proof of Theorem~\ref{c+euclidean}. For the same choice of $y_1$ and $y_2$, it is
easy to verify that
%
\[
  j_D(x,y_1) = \log \left( 1+\frac{|x-y_1|}{|x-a|-|x-y_1|} \right) =  \log \left( 1+t|x-a| \right),
\]
which shows that $r$ is best possible. Similarly, it can be verified that
\[
  j_D(x,y_2) = \log \left( 1+\frac{|x-y_2|}{|x-a|+|x-y_2|} \right) = \log \left( \frac{1}{1-t|x-a|} \right),
\]
which shows that $R$ is best possible.
\end{proof}

Under the light of Theorem~\ref{c+euclidean+punctured}, for an arbitrary proper subdomain 
$D$ of $\Rn$, we conjecture that

\begin{conjecture}
Let $D\subsetneq \Rn$, $x\in D$ and $t>0$. Then the following holds:
$$B_j(x,r)\subset B_c(x,t)\subset B_j(x,R),
$$
where $r=\ds\log (1+t\delta_D(x))$ 
and $R=\log\left(\ds\frac{1}{1-t\delta_D(x)}\right)$. 
Moreover, the radii $R$ and $r$ are best possible and $R/r\to 1$ as $t\to 0$.
\end{conjecture}

Next, we discuss the inclusion relation (\ref{inclusion}) when $d$
is the hyperbolic metric of the unit ball $\Bn$.
In the unit ball $\Bn$, the $j_{\Bn}$ metric and the $\rho_{\Bn}$ metric are 
comparable and is given by the relation 
\begin{equation}\label{rhojbn}
j_{\Bn}(x,y)\le \rho_{\Bn}(x,y)\le 2 j_{\Bn}(x,y)
\end{equation}
for $x,y\in \Bn$; see \cite[Theorem~7.56]{AVV97}.
The second inequality reduces to equality when $y=-x$.
It immediately follows that for $x\in \Bn$ and $r>0$, 
$$B_{\rho}(x,r)\subset B_{j}(x,r)\subset B_{\rho}(x,2r).
$$
This leads to the following:
\begin{theorem}\label{c+rho}
Let $x\in \Bn$ and $t>0$. Then the following inclusion relation holds:
$$B_{\rho}(x,r)\subset B_c(x,t)\subset B_{\rho}(x,R)
$$
where $r=\log\left(1+\ds\frac{t(1-|x|)}{1+t(1-|x|)}\right)$ and $R=\ds\frac{2t(1-|x|)}{1-t(1-|x|)}$. Moreover, $R/r\to 2$ as $t\to 0$.
\end{theorem}
\begin{proof}
By Theorem~\ref{c-j inclusion}, $B_c(x,t)\subset B_j(t(1-|x|)/(1-t(1-|x|)))$ and by 
(\ref{rhojbn}), the second inclusion follows with $R=2t(1-|x|))/(1-t(1-|x|))$. Again from
(\ref{rhojbn}) and Theorem~\ref{c-j inclusion}, we have
$$B_{\rho}(x, r)\subset B_{j}(x,r)\subset B_{c}(x,(e^r-1)/(1-|x|)(2-e^r)).
$$
On simplifying, we get $B_{\rho}(x,r)\subset B_{c}(x,t)$ with $r=\log(1+(t(1-|x|)/(1+t(1-|x|))))$. By l'H$\hat{\mbox{o}}$spital rule it is easy to see that 
$$\lim_{t\to 0}\frac{R}{r}=2.
$$
This completes the proof of our theorem.
\end{proof}

Another sharp inclusion property between $j$-metric ball and 
hyperbolic metric ball in $\Bn$ is derived by Kl\'en and Vuorinen in \cite{KV13}. Indeed, they 
proved that

\begin{lemma}\cite[Theorem~3.1]{KV13}\label{3.1kv13}
Let $x\in \Bn$ and $r>0$. Then
$$B_j(x,m)\subset B_{\rho}(x,r)\subset B_j(x,M)
$$
where $m=\max\{m_1,m_2\}$ and $M=\log\left(1+(1+|x|)\ds\frac{e^r-1}{2}\right)$;
$$m_1=\log\left(1+(1+|x|)\sinh(r/2)\right) \quad m_2=\log\left(1+(1-|x|)\frac{e^r-1}{2}
\right). 
$$
Moreover, the inclusions are sharp and $M/m\to 1$ as $r\to 0$.
\end{lemma}

Using Lemma~\ref{3.1kv13} together with Theorem~\ref{c-j inclusion} we obtain
\begin{theorem}\label{crhobn}
Let $x\in \Bn$ and $t>0$. Then the following inclusion relation holds:
$$B_{\rho}(x,r)\subset B_c(x,t)\subset B_{\rho}(x,R)
$$
where $r=\log\left(1+\ds\frac{2t(1-|x|)}{(1+|x|)(1+t(1-|x|))}\right)$ and $R=\min\{R_1,R_2\}$
with
$$R_1=2\arsh\left(\frac{\exp\left(\frac{t(1-|x|)}{1-t(1-|x|)}-1\right)}{1+|x|}\right)\quad
~~\mbox{ and }~~
R_2=\log\left(1+\frac{2\exp\left(\frac{t(1-|x|)}{1-t(1-|x|)}-1\right)}{1-|x|}\right).
$$
\end{theorem}

\begin{remark}
If $R=R_1$, then Theorem~\ref{crhobn} is sharper than Theorem~\ref{c+rho} 
(since $R_1/r\to 1$ as $t\to 0$). Otherwise Theorem~\ref{c+rho} is sharper than Theorem~\ref{crhobn}. 
\end{remark}

However, we conjecture a better estimate for radii $r$ and $R$ in Theorem~\ref{c+rho}.

\begin{conjecture}
Let $x\in \Bn$ and $t>0$. Then the following inclusion relation holds:
$$ B_\rho(x,r) \subset B_c(x,t)\subset B_\rho(x,R)
$$ 
where $r=t(1-|x|)/\sqrt{(1+|x|)(1+|x|-2t(1-|x|))}$ 
and \newline $R=t(1-|x|)/\sqrt{(1+|x|)(1+|x|+2t(1-|x|)}$. Moreover, the radii $r$ and $R$ are 
sharp and $R/r\to 1$ as $t\to 0$.
\end{conjecture}

In order to discuss the relation (\ref{inclusion}) when $d=k_D$, for a domain $D\subsetneq\Rn$, 
we recall the useful inequality \cite[Lemma~2.1]{GP76}
\begin{equation}\label{jkD}
k_D(x,y)\ge j_D(x,y); \quad x,y\in D\subsetneq \Rn.
\end{equation}
It follows immediately from (\ref{jkD}) and Theorem~\ref{c+euclidean+punctured} that
\begin{corollary}\label{newcor}
For $a\in \Rn$, let $D=\Rn\setminus\{a\}$, $x\in D$ and $t>0$. 
Then the following holds:
$$B_k(x,r)\subset B_c(x,t)
$$
where $r=\ds\log (1+t|x-a|)$.
\end{corollary}

At present we do not have any proof for sharpness of the inclusion relation 
in Corollary~\ref{newcor}. Therefore, it is appropriate here to state the 
following conjecture:
\begin{conjecture}
For $a\in \Rn$, let $D=\Rn\setminus\{a\}$, let $x\in D$ and $t>0$.
Then the following inclusion relation holds:
$$B_{k}(x,r)\subset B_c(x,t)\subset B_k(x,R)
$$
where $r=\log(1+t|x-a|)$ and $R=\log\left(\ds\frac{1}{1-t|x-a|}\right)$. The radii $r$ and $R$
are sharp and $R/r\to 1$ as $t\to 0$.
\end{conjecture}

In proper subdomains of $\Rn$ the following inclusion relation holds in between the 
Cassinian metric ball and the quasihyperbolic ball. 
The following lemma is useful in this setting.
\begin{lemma}\cite[Proposition~2.2]{KV12}\label{jkInclusion_general}
Let $D\subsetneq \Rn$ be a domain and $r\in (0,\log 2)$. Then 
$$B_j(x,m)\subset B_k(x,r)\subset B_j(x,r)\subset B_k(x,M)
$$
where $r=\log(2-e^r)$ and $M=\log\left(\ds\frac{1}{2-e^r}\right)$. Moreover, the second
inclusion is sharp and $M/m\to 1$ as $r\to 0$.
\end{lemma}

\begin{theorem}
Let $D\subsetneq \Rn$ be a domain, $x\in D$, and $t\in \left(0,\frac{\log 2}{\delta_D(x)(1+\log 2)}\right)$.
Then 
$$B_k(x,r)\subset B_c(x,t)\subset B_k(x,R)
$$
where $r=\log\left(1+\ds\frac{t\delta_D(x)}{1+t\delta_D(x)}\right)$ and 
$R=\log\left(\ds\frac{1}{2-\exp\big (\frac{t\delta_D(x)}{1-t\delta_D(x)}\big)}\right)$.
Moreover, $R/r\to 1$ as $t\to 0$.
\end{theorem}
\begin{proof}
By (\ref{jkD}) and Theorem~\ref{c-j inclusion} we have
$$B_k(x,r)\subset B_j(x,r)\subset B_c(x,(e^r-1)/(\delta_D(x)(2-e^r)))
$$
and the first inclusion follows. Again from Theorem~\ref{c-j inclusion} and Lemma~\ref{jkInclusion_general} we have
$$
  B_c(x,t)\subset B_j \left( x, \frac{t\delta_D(x)}{1-t\delta_D(x)} \right) \subset B_k \left( x,\log \left( \frac{1}{2-\exp(t\delta_D(x)/(1-t\delta_D(x)))} \right) \right).
$$
By l'H$\hat{\mbox{o}}$spital rule it follows easily that $R/r\to 1$ as $t\to 0$.
Hence the proof of our theorem is complete. 
\end{proof}

\medskip
\begin{acknowledgement}
The authors would like to thank Professor Zair Ibragimov for useful discussions on this
topic.
\end{acknowledgement}


\end{document}